\documentclass[11pt, a4paper, english]{amsart}
\usepackage{amsmath} % utilities for mathematics
\usepackage{amsthm} % utilities for theorem environment
\usepackage{amssymb} % loads mathematical fonts
\usepackage[dvipsnames]{xcolor}
\usepackage{amscd} % utilities for commutative diagrams
\usepackage{mathtools}
\usepackage{booktabs}
\usepackage[boxed]{algorithm2e}
\usepackage{subcaption}% <-- added

\usepackage{array} % core package
\usepackage[cal=boondoxo,scr=euler]{mathalfa}
\usepackage[backref=page,linktocpage]{hyperref} % permits navigating on the pdf
\usepackage{cleveref} % smart referencing
\usepackage{caption} %
\usepackage{graphics,graphicx} % permits putting images
\usepackage{tikz,tikz-cd} % tool for diagrams

\usetikzlibrary{graphs,graphs.standard,calc}

\usetikzlibrary{decorations.pathreplacing,}%angles,quotes

\usepackage{enumerate} % permits enumerating using letters or other symbols

\DeclareMathAlphabet{\mathsf}{OT1}{\sfdefault}{m}{n}

\newcommand{\nocontentsline}[3]{}
\newcommand{\tocless}[2]{\bgroup\let\addcontentsline=\nocontentsline#1{#2}\egroup}

\usepackage[margin=1.30in]{geometry} 
\usepackage{verbatim}

\usepackage{scalerel}

\makeatletter
\def\dual#1{\expandafter\dual@aux#1\@nil}
\def\dual@aux#1/#2\@nil{\begin{tabular}{@{}c@{}}#1\\#2\end{tabular}}
\makeatother

\makeatletter
\@namedef{subjclassname@2020}{\textup{2020} Mathematics Subject Classification}
\makeatother

\DeclareMathAlphabet{\amathbb}{U}{bbold}{m}{n}

\hypersetup{
    colorlinks = true,
    linkbordercolor = {white},
    linkcolor = {NavyBlue},
    anchorcolor = {black},
    citecolor = {NavyBlue},
    filecolor = {cyan},
    menucolor = {NavyBlue},
    runcolor = {cyan},
    urlcolor = {NavyBlue}
}

\usetikzlibrary{automata}

\newtheoremstyle{teoremas}% <name>
{12pt}% <Space above>
{13pt}% <Space below>
{\itshape}% <Body font>
{}% <Indent amount>
{\bfseries}% <Theorem head font>
{}% <Punctuation after theorem head>
{.5em}% <Space after theorem headi>
{}% <Theorem head spec (can be left empty, meaning `normal')>

\theoremstyle{teoremas}
\newtheorem{theorem}{Theorem}[section]
\newtheorem{corollary}[theorem]{Corollary}

\newtheorem{proposition}[theorem]{Proposition}

\newtheoremstyle{definition}% <name>
{12pt}% <Space above>
{12pt}% <Space below>
{}% <Body font>
{}% <Indent amount>
{\bfseries}% <Theorem head font>
{}% <Punctuation after theorem head>
{.5em}% <Space after theorem headi>
{}% <Theorem head spec (can be left empty, meaning `normal')>

\theoremstyle{definition}

\newtheorem{remark}[theorem]{Remark}

\crefname{theorem}{theorem}{theorems}
\Crefname{theorem}{Theorem}{Theorems}
\crefname{lemma}{lemma}{lemmas}
\Crefname{lemma}{Lemma}{Lemmas}
\crefname{proposition}{proposition}{propositions}
\Crefname{proposition}{Proposition}{Propositions}

\AtBeginDocument{%
   \def\MR#1{}
}

\title[Symmetric edge polytopes are not $\gamma$-positive]{Symmetric edge polytopes are not gamma-positive}

\author[L.~Ferroni]{Luis Ferroni}
\address{(L. Ferroni)
  Dipartimento di Matematica, Universit\`a di Pisa, Pisa, Italy.
}
\email{luis.ferroni@unipi.it}

\allowdisplaybreaks

\begin{document}

\begin{abstract}
    A conjecture posed by Ohsugi and Tsuchiya (2019) postulates that the Ehrhart $h^*$-polynomials of symmetric edge polytopes are $\gamma$-positive. We disprove this conjecture by exhibiting an infinite family of counterexamples. The smallest example provided by our construction is a $36$-dimensional symmetric edge polytope.
\end{abstract}

\maketitle

{\hfill\footnotesize\emph{Para Margarita y Bruno.}}

\section{Introduction}\label{sec:one}

\subsection{Overview}

Let $G$ be a graph with vertex set $V(G) = \{1,\ldots,n\}$ and edge set $E(G)$. The \emph{symmetric edge polytope} of $G$, denoted by $\mathcal{Q}_G$, is defined by
    \[\mathcal{Q}_G = \text{convex hull}\{ \pm (e_i-e_j) : \text{$ij\in E(G)$}\} \subseteq \mathbb{R}^n.\]
These polytopes were first considered in \cite{matsui-hibi-higashitani-nagazawa-hibi} and have attracted considerable attention in the community over the past decade, especially from the perspective of Ehrhart theory (see, e.g., \cite{higashitani-jochemko-michalek,higashitani-kummer-michalek,ohsugi-tsuchiya-conj,ohsugi-tsuchiya-2020-2,dali-juhnke-venturello,kalman-tothmeresz0,kalman-tothmeresz1,codenotti-riccardi-venturello}). 

Symmetric edge polytopes have important polyhedral properties. Notably, they are reflexive, and hence their Ehrhart $h^*$-vectors are palindromic. Furthermore, these polytopes possess regular unimodular triangulations (see \cite{higashitani-jochemko-michalek}); thus, by a result of Bruns and R\"omer \cite{bruns-romer}, the Ehrhart $h^*$-polynomial of a symmetric edge polytope is unimodal and coincides with the $h$-polynomial of a simplicial polytope.

A natural strengthening of unimodality for a palindromic polynomial is $\gamma$-positivity, a concept introduced in the work of Gal \cite{gal} and Br\"and\'en \cite{branden-gamma} in the context of the Charney--Davis and Neggers--Stanley conjectures. In this setting, based on extensive computations for graphs with a small number of vertices and for graphs belonging to special families, Ohsugi and Tsuchiya \cite[Conjecture~5.11]{ohsugi-tsuchiya-conj} proposed that $\gamma$-positivity holds for these polytopes.

Ohsugi and Tsuchiya's conjecture has sparked several positive results, among which it is worth emphasizing the proofs of positivity of the linear and quadratic coefficients of the $\gamma$-polynomial of symmetric edge polytopes. For the linear term, we refer to \cite{dali-juhnke-venturello,kalman-tothmeresz0}, whereas for the quadratic term we refer to \cite{dali-juhnke-venturello,codenotti-riccardi-venturello}. Furthermore, Ohsugi and Tsuchiya proved $\gamma$-positivity for a rather large class of graphs: namely, graphs possessing a vertex connected to every other vertex (i.e., \emph{cones} over graphs). Higashitani, Jochemko, and Micha{\l}ek \cite{higashitani-jochemko-michalek} proved it for complete bipartite graphs.

Despite the aforementioned evidence toward a positive resolution of Ohsugi and Tsuchiya's conjecture, we produce an infinite family of counterexamples.

\begin{theorem}
    There exist graphs $G$ for which the $h^*$-polynomial of $\mathcal{Q}_G$ is not $\gamma$-positive.
\end{theorem}

The examples we propose are series-parallel graphs, constructed as follows. Fix $m\geq 1$, and consider any vector $\mathbf{a}=(a_1,\ldots,a_m)\in \mathbb{Z}_{\geq 1}$. Let $G(\mathbf{a})$ be the graph obtained by joining two distinguished vertices $s$ and $t$ with $m$ disjoint paths, each of which has length $a_i$ for $i=1,\ldots,m$. For example, if $m = 5$ and $a_1=a_2=a_3=a_4=a_5=8$, the graph is depicted in Figure~\ref{fig:graph}. This graph is precisely the smallest example for which we have been able to detect the failure of $\gamma$-positivity. By increasing the length of the vector $\mathbf{a}$, or by increasing its entries, it is not hard to extend this construction to higher dimensions.

\begin{figure}[ht]
    \centering
    \begin{tikzpicture}
    [scale=0.72,auto=center,every node/.style={circle,scale=0.8, fill=black, inner sep=2.3pt}]
    \tikzstyle{edges} = [thick];

    \node[] (s) at (0,0) {};
    \node[] (t) at (9.6,0) {};

    \foreach \r/\y in {1/3.2,2/1.6,3/0,4/-1.6,5/-3.2} {
        \foreach \i in {1,...,7} {
            \node[] (v\r-\i) at ({1.2*\i},\y) {};
        }
        \draw[edges] (s) -- (v\r-1);
        \foreach \i/\j in {1/2,2/3,3/4,4/5,5/6,6/7} {
            \draw[edges] (v\r-\i) -- (v\r-\j);
        }
        \draw[edges] (v\r-7) -- (t);
    }

    \node[fill=white,scale=1.15,inner sep=1pt] at (-0.45,0) {$s$};
    \node[fill=white,scale=1.15,inner sep=1pt] at (10.05,0) {$t$};
    \end{tikzpicture}
    \caption{The graph $G(8,8,8,8,8)$.}
    \label{fig:graph}
\end{figure}
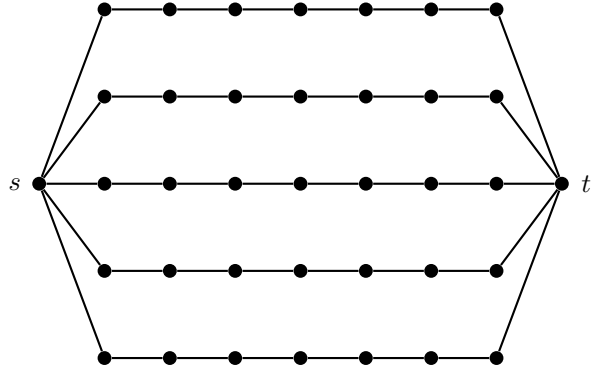

The natural question is how to compute the $h^*$-polynomials of these graphs. Although reasonably compact formulas should be possible for arbitrary vectors $\mathbf{a} \in \mathbb{Z}_{\geq 1}$, for the purposes of the counterexample it suffices to focus on the case in which all entries of $\mathbf{a}$ are equal. The main computational tool we employ is a formula of K\'alman and T\'othm\'er\'esz \cite{kalman-tothmeresz0}, which writes the $h^*$-vector of an arbitrary symmetric edge polytope in terms of any dissecting tree set of the graph. Because of their structure, our graphs are relatively manageable, and one can derive closed formulas that ultimately verify the main result of the paper.

\subsection*{How this counterexample was found} Over the past two years, especially since the disproof of a more general conjecture by Davis, Higashitani, and Ohsugi \cite{davis-higashitani-ohsugi} on $\gamma$-positivity for symmetric edge polytopes of regular matroids, I have encouraged many colleagues to carry out computations for series-parallel graphs of similar flavor to the ones appearing in the present manuscript. It was not until a few days ago, when Matt Larson shared with me a preliminary version of his recent double breakthrough \cite{larson} (to appear on the arXiv today), that I found enough courage to carry out the computation appearing as Proposition~\ref{prop:formula} below. This discovery was made essentially by hand, with AI assistance only in the verification of Proposition~\ref{prop:formula}. The formula was verified by GPT-5.5 Pro in two distinct ways: using the definitions in \cite{kalman-tothmeresz0} directly, and using the Gr\"obner basis of \cite{higashitani-jochemko-michalek} (see \href{https://chatgpt.com/share/6a4695ed-cd94-83eb-af72-5c3f7ca0f4a8}{here} for the verification of the main counterexample).

\section{The proof}

Recall that for $m\geq 1$ and $\mathbf{a}=(a_1,\ldots,a_m)\in \mathbb{Z}_{\geq 1}$, the graph $G(\mathbf{a})$ is constructed by joining two distinguised vertices $s$ and $t$ with $m$ disjoint paths, each of which has length $a_i$ for $i=1,\ldots,m$. 

\begin{proposition}\label{prop:formula}
Fix integers $k\geq 2$ and $m\geq 1$, and let
\[
  G_{k,m}=G(\underbrace{k,\ldots,k}_{m\text{ times}}).
\]
Let $d=\dim \mathcal{Q}_{G_{k,m}}=m(k-1)+1$.
For $0\leq i\leq k+1$, define
\[
  P_{k,i}(t)=
  \sum_{r=0}^{i-1}\sum_{j=0}^{k-i}
  \binom{r+j}{r}\binom{k-1-r-j}{i-1-r}
  t^{k-i+r-j}.
\]
Then the $h^*$-polynomial
of the symmetric edge polytope $\mathcal{Q}_{G_{k,m}}$ is
\begin{equation}\label{eq:formula}
  h^*_{G_{k,m}}(t)=
  \sum_{i=0}^{k} \binom{k}{i}t^i
  \sum_{\ell=0}^{m-1}
  \bigl(P_{k,i}(t)\bigr)^\ell
  \bigl(P_{k,i+1}(t)\bigr)^{m-1-\ell}.
\end{equation}
\end{proposition}

\begin{proof}
We use the spanning-tree dissection formula of K\'alm\'an--T\'othm\'er\'esz
\cite[Theorem~1.1]{kalman-tothmeresz0}: if $\mathcal T$ is any dissecting tree set
for a connected graph $G$, and if a base vertex $v$ is fixed, then
\begin{equation}
  h_r^*(\mathcal{Q}_G)=\#\{T\in\mathcal T: T\text{ has exactly }r
  \text{ edges pointing away from }v\}.
\label{eq:second-formula}
\end{equation}
Here an edge of an oriented tree $T$ points away from $v$ if, after deleting that
edge from the underlying unoriented tree, the component containing $v$ contains the
tail of the directed edge.  We shall use this formula with $v=s$.

Let $\varphi$ be a facet functional for $\mathcal{Q}_{G_{k,m}}$.  By the usual facet
description of symmetric edge polytopes, $|\varphi(u)-\varphi(v)|\leq 1$ for every
edge $uv$, and the tight-edge subgraph is connected and spanning.  Since
$G_{k,m}$ is bipartite, the values of $\varphi$ on the two bipartition classes have
opposite parity: this is clear on the connected tight-edge subgraph, where adjacent
vertices differ by $\pm1$.  Hence every edge of $G_{k,m}$ joins two vertices whose
$\varphi$-values have opposite parity.  Since their difference has absolute value at
most one, every edge is tight.

Therefore a facet is the same thing as an orientation of every edge, subject only to
the condition that this orientation comes from a potential $\varphi$.  Write a sign
word
\[
  w=(w_1,\ldots,w_k)\in\{+,-\}^k
\]
for each path, where $+$ means that the corresponding edge is oriented from $s$ to
$t$, and $-$ means that it is oriented from $t$ to $s$.  If the $q$-th path has
$i_q$ plus signs, then along that path
\(
  \varphi(t)-\varphi(s)=i_q-(k-i_q)=2i_q-k.
\)
This number must be independent of the path.  Hence all the numbers $i_q$ are
equal.  Conversely, if all paths have the same number $i$ of plus signs, then the
path sums agree, so the signs integrate to a well-defined potential.  Thus facets are
indexed by choices of sign words on the $m$ paths having a common number
$i$ of plus signs.

We now analyze the tree dissection inside a facet.
Embed $G_{k,m}$ in the plane with the $m$ paths cyclically ordered.  Fix one of the
facet orientations described above.  The planar-dual arborescence construction for
root polytopes gives a dissecting tree set for this facet, and coning these facet
dissections from the origin gives a dissecting tree set for $\mathcal{Q}_{G_{k,m}}$. The planar dual has $m$ face-vertices arranged in a cycle; each of the $m$ sides of
this dual cycle is a bundle of $k$ parallel dual edges, one for each edge of the
corresponding primal path.  A dual spanning tree chooses one dual edge from
$m-1$ of these bundles and chooses no edge from the remaining bundle.  In the
primal complement this says exactly: one of the $m$ paths is kept whole, while on
each of the other $m-1$ paths exactly one edge is deleted.

Root the dual arborescence at the exterior face.  Once the whole primal path is
chosen, the broken paths split into two cyclic blocks.  For the paths in one block
the deleted edge is required to have sign $+$, while for the paths in the other block
it is required to have sign $-$.  If $\ell$ is the number of broken paths in the first
block, then $\ell$ runs through $0,1,\ldots,m-1$ as the whole path runs through
the $m$ possible choices.  The choice of clockwise versus counterclockwise
convention only interchanges the two signs, and the final sum over all $\ell$ is
unchanged.

Let us now count edges pointing away from $s$.
Fix the common number $i$ of plus signs.  If a path is kept whole, then the path
from $s$ to any edge of it travels from $s$ toward $t$.  Thus precisely the plus edges
of that path point away from $s$.  The contribution of the whole path is therefore $\binom{k}{i}t^i$,
because there are $\binom{k}{i}$ sign words with $i$ plus signs.

Now consider a broken path, and let $q$ be the deleted edge.  The part of the path
before $q$ is attached to $s$, so on this part the plus edges point away from $s$.
The part after $q$ is attached to the rest of the tree through $t$ and the unique
whole path; therefore, on this part, the minus edges point away from $s$.  For a
sign word $w$, the number of edges pointing away from $s$ on this broken path is
\[
  \alpha(w,q)=\#\{r<q:w_r=+\}+\#\{r>q:w_r=-\}.
\]
If the deleted edge is required to have sign $+$, the generating function for this
broken path is $P_{k,i}(t)$.  Indeed, let $r$ be the number of plus signs before the
deleted edge and let $j$ be the number of minus signs before it.  Then the exponent
is $r+(k-i-j)=k-i+r-j$,
and the number of sign words with these values of $r$ and $j$ is $\binom{r+j}{r}\binom{k-1-r-j}{i-1-r}$.
Summing over $0\leq r\leq i-1$ and $0\leq j\leq k-i$ gives exactly
$P_{k,i}(t)$.

Similarly, if the deleted edge is required to have sign $-$, the generating function
is $P_{k,i+1}(t)$.  In this case, with the same meanings of $r$ and $j$, the
exponent is $r+(k-i-1-j)=k-i-1+r-j$,
and the number of sign words with these values of $r$ and $j$ is
$\binom{r+j}{r}\binom{k-1-r-j}{i-r}$.
Summing over $0\leq r\leq i$ and $0\leq j\leq k-i-1$ gives exactly the polynomial
$P_{k,i+1}(t)$, by the definition of $P_{k,i+1}$.

Putting the pieces together, for fixed $i$ and fixed $\ell$, the generating function
for the number of edges pointing away from $s$ is
\[
  \binom{k}{i}t^i
  \bigl(P_{k,i}(t)\bigr)^\ell
  \bigl(P_{k,i+1}(t)\bigr)^{m-1-\ell}.
\]
Summing over $\ell=0,\ldots,m-1$ and $i=0,\ldots,k$, and then applying
\eqref{eq:second-formula}, gives \eqref{eq:formula}.
\end{proof}

\begin{corollary}
For $G_{8,5}=G(8,8,8,8,8)$, the above formula gives $d=36$ and
\[
\begin{aligned}
(h^*_0,\ldots,h^*_{36})=(&1,44,950,13420,139545,1139424,7611536,42797920,\\
&206716612,869517648,3208517832,10401251024,29530579604,\\
&72992503712,155919966064,285492984352,444071652750,\\
&581531566920,636773179140,581531566920,444071652750,\\
&285492984352,155919966064,72992503712,29530579604,\\
&10401251024,3208517832,869517648,206716612,42797920,\\
&7611536,1139424,139545,13420,950,44,1).
\end{aligned}
\]
Consequently, one obtains
\[
\begin{aligned}
(\gamma_0,\ldots,\gamma_{18})=(&1,8,48,256,1280,6144,28672,131072,461136,\\
&1373568,3438592,8290304,16487040,37120000,63755264,\\
&77045760,-9799680,0,0).
\end{aligned}
\]
In particular, $\gamma_{16}=-9799680$.
\end{corollary}

\begin{remark}
    It is easy to find examples with many more negative signs among the $\gamma$-coefficients. For example, for the graph with $12$ paths of length $12$, the sign pattern of $(\gamma_0,\gamma_1,\ldots, \gamma_{66})$ is:
   \[
\begin{array}{ccccccccccc}
(+, &+, &+, &+, &+, &+, &+, &+, &+, &+, &+,\\
+, &+, &+, &+, &+, &+, &+, &+, &+, &+, &+,\\
+, &+, &+, &+, &+, &+, &+, &+, &-, &+, &+,\\
-, &+, &+, &-, &+, &-, &+, &-, &+, &-, &+,\\
-, &+, &-, &+, &-, &+, &-, &+, &-, &+, &-,\\
+, &-, &+, &-, &+, &0, & 0,& 0,& 0,& 0)
\end{array}
\]
\end{remark}

\subsection*{Acknowledgments}

I am a member of the GNSAGA group of the INdAM. I want to thank Alessio D'Al\`i, Giulia Codenotti, Michele D'Adderio, and Lorenzo Venturello for numerous conversations about symmetric edge polytopes, especially in connection with the $\gamma$-positivity conjecture. My interest in this specific problem was inspired by those conversations. I also want to thank Matt Larson for sharing with me a preliminary version of \cite{larson}, which gave me the necessary push to pursue the computation carried out in this paper. This paper was written under very special circumstances, as I am expecting my daughter Margarita to be born in very few days. I plan to update the paper in the near future by adding more background and expanding the historical discussion around this problem and the related results.

\bibliographystyle{amsalpha}
\bibliography{bibliography}

\end{document}